\documentclass{amsart}
\usepackage[a4paper]{geometry}
\usepackage{latexsym}
\usepackage{amsmath}
\usepackage{amssymb}
\usepackage{amsthm}
\usepackage{bbm}
\usepackage{amsfonts}
\usepackage{mathrsfs}
\usepackage{cite}
\usepackage{color}
\usepackage{enumerate}

\newtheorem{thm}{Theorem}
\numberwithin{thm}{section}

\newcommand{\rubrik}{}
\newtheorem{prop}[thm]{Proposition}

\newtheorem{lem}[thm]{Lemma}

\theoremstyle{definition}

\newtheorem{defn}[thm]{Definition}

\theoremstyle{remark}

\newtheorem{rem}[thm]{Remark}              

\newcommand{\Ker}{\operatorname{Ker}}

\newcommand{\pd}[1] {\partial ^#1}
\newcommand{\pdd}[2] {\partial_{#1} ^{#2}}

\newcommand{\ro}{\mathbb R}
\newcommand{\no}{\mathbb N}
\newcommand{\rr}[1]{\mathbb R^{#1}}
\newcommand{\nn}[1]{\mathbb N^{#1}}

\newcommand{\co}{\mathbb C}

\newcommand{\dd}{\mathrm {d}}

\newcommand{\fy}{\varphi}

\newcommand{\wpr}{{\text{\footnotesize $\#$}}}

\newcommand{\eabs}[1]{\langle #1\rangle}

\newcommand{\Sp}{\operatorname{Sp}}
\newcommand{\Mp}{\operatorname{Mp}}
\newcommand{\GL}{\operatorname{GL}}
\newcommand{\M}{\operatorname{M}}

\newcommand{\cI}{\mathscr{I}}

\newcommand{\dbar}{{{{\ \mathchar'26\mkern-12mu \mathrm d}}}}

\newcommand{\WF}{\mathrm{WF}}

\newcommand{\diag}{\mathrm{diag}}
\newcommand{\dist}{\operatorname{dist}}

\newcommand{\cS}{\mathscr{S}}
\newcommand{\cT}{\mathcal{T}}
\newcommand{\cV}{\mathcal{V}}
\newcommand{\cTp}{\mathcal{T}_{\psi_0}}

\newcommand{\cK}{\mathscr{K}}
\newcommand{\cL}{\mathscr{L}}
\newcommand{\cR}{\mathscr{R}}

\newcommand{\J}{\mathcal{J}}

\def\la{\langle}
\def\ra{\rangle}
\newcommand{\leqs}{\leqslant}
\newcommand{\geqs}{\geqslant}

\theoremstyle{definition}

\theoremstyle{remark}

\numberwithin{equation}{section}

\begin{document}

\title{Shubin type Fourier integral operators and evolution equations}

\author{Marco Cappiello}
\address{Department of Mathematics, University of Torino, Via Carlo Alberto 10, 10123 Torino, Italy.}
\email{marco.cappiello[AT]unito.it}

\author{Ren\'e Schulz}
\address{Leibniz Universit\"at Hannover, Institut f\"ur Analysis, Welfenplatz 1, D--30167 Hannover, Germany}
\email{rschulz[AT]math.uni-hannover.de}
\thanks{R. Schulz gratefully acknowledges support of the project ``Fourier Integral Operators, symplectic geometry and analysis on noncompact manifolds'' received by the University of Turin in form of an ``I@Unito'' fellowship as well as institutional support by the University of Hannover.
}

\author{Patrik Wahlberg}
\address{Department of Mathematics, Linn{\ae}us University, SE--351 95 V\"axj\"o, Sweden}
\email{patrik.wahlberg[AT]lnu.se}





\begin{abstract}
We study the Cauchy problem for an evolution equation of Schr\"odinger type. 
The Hamiltonian is the 
Weyl quantization of a real homogeneous quadratic form with a pseudodifferential perturbation of negative order from Shubin's class. 
We prove that the propagator is a Fourier integral operator of Shubin type of order zero. 
Using results for such operators and corresponding Lagrangian distributions, 
we study the propagator and the solution, and derive phase space estimates for them. 
\end{abstract}

\maketitle

\section{Introduction}

In this article we study the propagator and solution to the Cauchy problem 
\begin{equation}
\label{eq:CP}
\tag{CP}
	\left\{
	\begin{array}{rl}
	\partial_t u(t,x) + i(q^w(x,D)+p^w(x,D)) u (t,x) & = 0, \qquad t > 0, \quad x \in \rr d, \\
	u(0,\cdot) & = u_0\in \cS'(\rr d),  
	\end{array}
	\right.
\end{equation}
where $q^w(x,D)$ is the Weyl quantization of a real homogeneous quadratic form on $T^* \rr d$ and $p^w(x,D)$ is 
a pseudodifferential perturbation operator 
with complex-valued Shubin type symbol $p$
of negative order. 
Particular examples of interest are perturbations to the free Schr\"odinger equation and the quantum harmonic oscillator. 

The Shubin class $\Gamma^m$, $m \in \ro$, introduced in \cite{Shubin1}, is defined as
the space of all functions $a \in C^\infty(\rr {2d})$ that satisfy estimates of the form
\begin{equation*}
|\partial^\alpha_x \partial^\beta_\xi a(x,\xi)| \lesssim (1+|x|+|\xi|)^{m-|\alpha+\beta|}, \qquad (x,\xi) \in \rr {2d}, \quad \alpha,\beta \in \nn d. 
\end{equation*}
Differently from H\"ormander symbols, the elements of $\Gamma^m$ exhibit a symmetric behavior in the decay with respect to $x$ and $\xi$. 
An interesting example is the symbol $a(x,\xi)=|x|^2+|\xi|^2 \in \Gamma^2$ for the harmonic oscillator operator. 
The theory of pseudodifferential operators with symbols in the Shubin classes has been developed in \cite{Shubin1} and widely applied to the study of several classes of partial differential equations, see e.g. \cite{Asada1, BBR, CGR, CN, CRT, Helffer1, HR1, Hormander1, Nicola1, Rodino1, PRW1, SW2, Tataru}. Helffer and Robert \cite{Helffer1, HR1} introduced Fourier integral operators (FIOs) with Shubin type amplitudes and phase functions that are generalized quadratic. 
Similar oscillatory integrals have been considered by Asada and Fujiwara \cite{Asada1}, see also \cite{BBR}. 

Concerning the Cauchy problem \eqref{eq:CP}, the case when $q(x,\xi)=|x|^2+|\xi|^2$ and $p=0$ is since long well known, see e.g. \cite{Folland1,deGosson2,Helffer1}. More generally, in the unperturbed case $p=0$, the solution operator to the equation \eqref{eq:CP} is a metaplectic operator, see e.g. \cite{Folland1}. 
Namely it is the unique one-parameter continuous group of metaplectic operators $\mu_t$, associated with the Hamiltonian flow $\chi_t = e^{2 t F}$ of $q$ and chosen such that $\mu_0=I$, 
where $F = \J Q$ is the real $2d \times 2d$ matrix determined by the symmetric matrix $Q$ defining $q$, $q(x,\xi) = \la (x,\xi), Q (x,\xi) \ra$, 
and the symplectic matrix
\begin{equation}\label{eq:Jdef}
\J = 
\left(
\begin{array}{cc}
0 & I_d \\
-I_d & 0 
\end{array}
\right). 
\end{equation}

We consider now the problem \eqref{eq:CP} under the presence of a non-vanishing perturbation $p$.
Recently the problem has been studied in \cite{CGNR} assuming the symbol $p$ belong to a weighted modulation space of Sj\"ostrand type, whose elements are not necessarily smooth, see also \cite{CNR,Weinstein1}. The authors proved that the equation \eqref{eq:CP} admits a propagator given by the composition of a metaplectic operator and a Weyl pseudodifferential operator with symbol in the same modulation space as $p$.

In this paper we prove  a similar statement in a different setting, namely the following result.

\begin{thm}\label{thm:mainresult}
If $\delta >0$ and $p \in \Gamma^{-\delta}$ then the Cauchy problem \eqref{eq:CP} has a propagator of the form $\mu_t a_t^w(x,D)$ where $a_t\in \Gamma^0$ for $t \geqs 0$. 
\end{thm}

With respect to \cite{CGNR} we assume more regularity on the symbol of the perturbation and we obtain a stronger conclusion on the regularity of $a_t$. Moreover, the fact that $a_t$ is a Shubin symbol allows us to obtain additional results in terms of propagation of singularities and phase estimates for the solution. For this we take advantage of some recent results for a class 
of FIOs with quadratic phase functions and Shubin amplitudes, cf. \cite{Cappiello2, Cappiello3}. In these papers
we proved phase space estimates for an FBI type transform of the kernels of the operators, see \cite{Tataru} for similar estimates in a particular case. We also proved that every operator in the class can be written as the composition of a metaplectic operator and a pseudodifferential operator with Shubin symbol and vice versa. 
As a byproduct of the analysis we derived a new notion of Lagrangian distributions in the Shubin framework which generalizes the properties of the kernels of FIOs.

Under the assumptions of Theorem \ref{thm:mainresult} the propagator  of \eqref{eq:CP} belongs to this class of FIOs for each $t \geqs 0$.
This opens up the possibility to study the singularities of solutions to \eqref{eq:CP} in detail, proving propagation results for Lagrangian type singularities and phase space estimates for the solution, see Theorems 4.5 and 4.9 below. 

The paper is organized as follows. 
In Section \ref{sec:prelim} we recall the technical tools for our analysis, in particular aspects of pseudodifferential quantization, metaplectic and symplectic analysis, Shubin type FIOs, and properties of an FBI type phase space transform which is a fundamental tool. In Section \ref{sec:application} we construct a parametrix to \eqref{eq:CP} and prove that the propagator is a Shubin type FIO. Finally in Section \ref{sec:sing} we study the singularities of propagators and solutions to \eqref{eq:CP} and deduce phase space estimates for them.
%

\section{Preliminaries on microlocal analysis in Shubin's class}\label{sec:prelim}

\subsection*{Basic notation}

The gradient operator with respect to $x \in \rr d$ is denoted $\nabla_x$. 
The symbols $\cS(\rr d)$ and $\cS'(\rr d)$ denote the Schwartz space of rapidly decaying smooth functions and the tempered distributions, respectively. The notation $f (x) \lesssim g(x)$ means $f(x) \leqs C g(x)$ for some $C>0$ for all $x$ in the domain of $f$ and of $g$. 
We write 
$(f,g)$ for the sesquilinear pairing, conjugate linear in the second argument, between a distribution $f$ and a test function $g$, as well as the $L^2$ scalar product if $f,g \in L^2(\rr d)$.
The linear pairing of a distribution $f$ and a test function $g$ is written $\la f, g \ra$. 
The symbols $T_{x_0}u(x)=u(x-x_0)$ and $M_\xi u(x) = e^{i \la x,\xi \ra}u(x)$, where $\la \cdot,\cdot \ra$ denotes the inner product on $\rr d$, are used for translation by $x_0\in \rr d$ and modulation by $\xi \in \rr d$, respectively, applied to functions or distributions. For $x\in\rr{d}$ we use $\eabs{x}:=\sqrt{1+|x|^2}$, and 
Peetre's inequality is
\begin{equation*}
\eabs{x+y}^s \leqs C_s\, \eabs{x}^s \,\eabs{y} ^{|s|}, \qquad x, y \in\rr{d}, \quad C_s>0, \quad s \in \ro. 
\end{equation*}
We write $\dbar x = (2\pi)^{-d}\dd x$ for the dual Lebesgue measure, 
denote by $\M_{d_1 \times d_2}( \ro )$ the space of $d_1\times d_2$ matrices with real entries, and by $\GL(d,\ro ) \subseteq \M_{d \times d}( \ro )$ the group of invertible matrices. 
The orthogonal projection on a linear subspace $Y \subseteq \rr d$ is denoted $\pi_Y$. 
The symbol $\cL(H)$ stands for the space of linear continuous operators on a Hilbert space $H$. 

\subsection*{An integral transform of FBI type} 

The following integral transform has been used extensively in \cite{Cappiello2,Cappiello3} and is used also in this article. 
For more information see \cite{Cappiello2}. 
\begin{defn}\label{def:FBItransform}
Let $u\in \cS^\prime(\rr d)$ and let $g\in \cS(\rr d)\setminus\{0\}$. The transform $u \mapsto \cT_g u$ is defined by 
\begin{equation*}
\cT_g u(x,\xi)=(2\pi)^{-d/2}(u,T_x M_{\xi}g), \quad x, \xi \in \rr d. 
\end{equation*}
\end{defn}

If $u \in \cS(\rr d)$ then $\cT_g u \in \cS(\rr {2d})$ by \cite[Theorem~11.2.5]{Grochenig1}. 
The adjoint $\cT_g^*$ is defined by $(\cT_g^* U, f) = (U, \cT_g f)$ for $U \in \cS'(\rr {2d})$ and $f \in \cS(\rr d)$. 
When $U$ is a polynomially bounded measurable function we write
\begin{equation*}
\cT_g^* U(y) = (2\pi)^{-d/2} \int_{\rr {2d}} U(x,\xi) \, T_{x} M_{\xi} g(y) \, \dd x \, \dd \xi ,
\end{equation*}
where the integral is defined weakly so that $(\cT_g^* U, f) = (U, \cT_g f)_{L^2}$ for $f \in \cS(\rr d)$. 
\begin{prop}
\label{prop:Swdchar}
{\rm \cite[Theorem~11.2.3]{Grochenig1}}
Let $u\in\cS'(\rr d)$ and let $g \in \cS(\rr d) \setminus 0$. Then $\cT_g u\in C^\infty(\rr {2d})$ and there exists $N \in \no$ such that %
\begin{equation*}
|\cT_g u(x,\xi)|\lesssim \eabs{(x,\xi)}^{N}, \quad (x,\xi) \in \rr {2d}.
\end{equation*}
We have $u\in \cS(\rr d)$ if and only if for any $N \geqs 0$ 
\begin{equation*}
|\cT_g u(x,\xi)|\lesssim \eabs{(x,\xi)}^{-N}, \quad (x,\xi) \in \rr {2d}. 
\end{equation*}
\end{prop}

The transform $\cT_g$ is related to the short-time Fourier transform \cite{Grochenig1}
\begin{equation*}
\cV_g u(x,\xi) = (2\pi)^{-d/2}(u,M_{\xi} T_x g), \quad x, \xi \in \rr d, 
\end{equation*}
viz. $\cT_g u(x,\xi) = e^{i \la x, \xi \ra} \cV_{g}u(x,\xi)$.
If $g,h\in \cS(\rr d)$ then
\begin{equation*}
\cT_h^*\cT_g u=(h,g) u, \qquad u \in \cS'(\rr d), 
\end{equation*}
and thus $\|g\|_{L^2}^{-2}\cT_g^*\cT_g u=u$ for $u \in \cS'(\rr d)$ and $g \in \cS(\rr d) \setminus 0$, cf. \cite{Grochenig1}.

Finally we recall the definition of the Gabor wave front set which describes global singularities of tempered distributions in phase space, cf. \cite{Hormander1,Rodino1,SW2}. 

\begin{defn}\label{def:WFG}
If $u \in \cS'(\rr d)$ and $g \in \cS(\rr d) \setminus 0$ then $z_0 \in T^*\rr d  \setminus 0$ satisfies  $z_0 \notin \WF(u)$ if 
there exists an open cone $V \subseteq T^* \rr d \setminus 0$ containing $z_0$, such that for any $N \in \no$ there exists $C_{V,g,N}>0$ such that $|\cT_g u(z)|\leqs C_{V,g,N} \eabs{z}^{-N}$ when $z \in V$.
\end{defn}

The Gabor wave front set is hence a closed conic subset of $T^*\rr d  \setminus 0$. 
If $u \in \cS'(\rr d)$ then $WF(u) = \emptyset$ if and only if $u \in \cS(\rr d)$ \cite[Proposition~2.4]{Hormander1}. 

\subsection*{Weyl pseudodifferential operators} 

We use pseudodifferential operators in the Weyl calculus with Shubin amplitudes \cite{Shubin1,Nicola1}. 
Recall that $a \in C^\infty(\mathbb{R}^{N_1}\times \mathbb{R}^{N_2})$ is a Shubin amplitude of order $m \in \ro$, 
denoted $a\in \Gamma^m(\mathbb{R}^{N_1}\times \mathbb{R}^{N_2})$, if it satisfies the estimates 
\begin{equation}\label{eq:shubinestimate}
|\partial_x^\alpha \partial_\xi^\beta a(x,\xi)| 
\lesssim \eabs{(x,\xi)}^{m-|\alpha+\beta|}, \quad (\alpha, \beta) \in \mathbb{N}^{N_1}\times \mathbb{N}^{N_2}, \quad (x,\xi) \in \mathbb{R}^{N_1}\times \mathbb{R}^{N_2}. 
\end{equation} 
We write $\Gamma^m = \Gamma^m(\rr {2d})$ and observe that $\bigcap_{m \in \ro} \Gamma^{m}=\cS(\rr {2d})$. 
The space $\Gamma^m$ is a Fr\'echet space with respect to the seminorms that are the best constants hidden in \eqref{eq:shubinestimate}.  

To a Shubin amplitude $a \in \Gamma^m$ one associates its pseudodifferential Weyl quantization, which is the operator $a^w(x,D)$ with Schwartz kernel
\begin{equation}\label{eq:weylsymbolkernel}
K_{a}(x,y) = \int_{\rr {d}} e^{i \la x-y, \xi \ra} a\left((x+y)/2,\xi \right) \, \dbar \xi \in\cS'(\rr {2d})
\end{equation}
interpreted as an oscillatory integral. 
Then $a^w(x,D)$ is a continuous operator on $\cS(\rr d)$ that extends uniquely to a continuous operator on $\cS'(\rr d)$.
If $a\in\cS(\rr {2d})$ then $a^w(x,D): \cS'(\rr d)\rightarrow \cS(\rr d)$ is continuous when $\cS'(\rr d)$ is equipped with its strong topology. 
Conversely, any continuous linear operator from $\cS'(\rr d)$, endowed with the strong topology, to $\cS(\rr d)$ may be represented as $a^w(x,D)$ for some $a\in\cS(\rr {2d})$ \cite{Treves1}. 

For $a \in \cS'(\rr {2d})$ and $f,g \in \cS(\rr d)$ we have
\begin{equation}\label{eq:wignerweyl}
(a^w(x,D) f,g) = (2 \pi)^{-d/2} (a, W(g,f) ) 
\end{equation}
where $W(g,f)$ is the Wigner distribution \cite{Folland1,Grochenig1}
\begin{equation*}
W(g,f) (x,\xi) = (2 \pi)^{-d/2} \int_{\rr d} g(x+y/2) \overline{f(x-y/2)} \, e^{- i \la y, \xi \ra} \, \dd y \in \cS(\rr {2d}).
\end{equation*}

The Weyl product $a \wpr b:\Gamma^{m_1}\times\Gamma^{m_2}\rightarrow \Gamma^{m_1+m_2}$ is the continuous product (cf. \cite{Shubin1}) on the symbol level corresponding to composition of operators:
\begin{equation*}
(a \wpr b )^w(x,D) = a^w(x,D) b^w(x,D).
\end{equation*}
There is a scale of Sobolev spaces $Q^s(\rr d)$, $s \in \ro$, defined by
\begin{equation*}
Q^s (\rr d) = \{ u \in \mathscr{S}'(\rr d): v_s^w(x,D) u \in L^2(\rr d) \}, 
\end{equation*}
where $v_s(x,\xi) = \langle (x,\xi) \rangle^s$, which is adapted to the Shubin calculus.
We have (cf. \cite[Corollary~25.2]{Shubin1})
\begin{equation}\label{eq:SQs}
\cS(\rr d)=\bigcap_{s \in \ro} Q^s(\rr d),\qquad \cS'(\rr d) = \bigcup_{s \in \ro} Q^s(\rr d).
\end{equation}
The Weyl quantization of $\Gamma^m$ yields continuous maps
\begin{equation}\label{eq:shubinpsdocont}
a^w(x,D):Q^{s}(\rr d)\rightarrow Q^{s-m}(\rr d), \qquad s\in\ro, 
\end{equation}
and the $Q^s\rightarrow Q^{s-m}$ operator norm of $a^w(x,D)$ can be estimated by a finite linear combination of seminorms of $a \in \Gamma^m$.

We use the description of $Q^s$ in terms of localization operators \cite[Proposition~1.7.12]{Nicola1}.
Let $\psi_0 = \pi^{-d/4} e^{-|x|^2/2}$, $x \in \rr d$. 
A localization operator $A_a$ with symbol $a \in \cS'(\rr {2d})$ is defined by 
\begin{equation*}
(A_a u, f) = (a \, \cTp u, \cTp f), \quad u,f \in \cS(\rr d). 
\end{equation*}
In terms of the localization operator $A_s := A_{v_s}$, the space $Q^s(\rr d)$ is the Hilbert modulation space of all $u\in\cS'(\rr d)$ such that $A_s u\in L^2(\rr d)$, equipped with the norm
$\| u \|_{Q^s} = \| A_s u \|_{L^2}$.

It is possible to express localization operators as pseudodifferential operators (cf. \cite[Section~1.7.2]{Nicola1}) writing $A_a = b^w(x,D)$ where 
\begin{equation}\label{eq:localizationweyl}
b = \pi^{-d} e^{-|\cdot|^2} * a. 
\end{equation}
%

\subsection*{Metaplectic operators}

We view $T^* \rr d \cong \rr d\times\rr d$ as a symplectic vector space equipped with the 
canonical symplectic form
\begin{equation}\label{eq:cansympform}
\sigma((x,\xi), (x',\xi')) = \la x' , \xi \ra - \la x, \xi' \ra, \quad (x,\xi), (x',\xi') \in T^* \rr d.
\end{equation}
The real symplectic group $\Sp(d,\ro) \subseteq \GL(2d,\ro)$ is the set of matrices that leaves $\sigma$ invariant. 
An often occurring symplectic matrix is $\J \in \Sp(d,\ro)$ defined in \eqref{eq:Jdef}. 

The metaplectic group \cite{deGosson2,Leray1} $\Mp(d)$ is a group of unitary operators on $L^2(\rr d)$,
which is a (connected) double covering of the symplectic group $\Sp(d,\ro)$. 
In fact the two-to-one projection $\pi: \Mp(d) \rightarrow \Sp(d,\ro)$ has kernel is $\pm I$. 
Each operator $\mu \in \Mp(d)$ is a homeomorphism on $\mathscr S$ and on $\mathscr S'$. 
The metaplectic covariance of the Weyl calculus reads 
\begin{equation}\label{eq:metaplecticoperator}
\mu^{-1} a^w(x,D) \, \mu = (a \circ \chi_\mu)^w(x,D), \quad a \in \cS'(\rr {2d}), 
\end{equation}
where $\mu \in \Mp(d)$ and $\chi_\mu = \pi (\mu)$ (cf. \cite[Theorem~215]{deGosson2}, \cite{Folland1}).

\subsection*{Fourier integral operators with Shubin amplitudes}

In \cite{Cappiello3} we have introduced a class of Fourier integral operators (FIOs) with quadratic phase functions and Shubin amplitudes. The space of Shubin type FIOs of order $m \in \ro$ associated with $\chi \in \Sp(d,\ro)$, denoted $\cI^m(\chi)$, consists of those operators $\cK$ whose kernels admit oscillatory integral representations of the form 
\begin{equation*}
K_{a,\varphi}(x,y) = \int_{\rr N} e^{i \varphi(x,y,\theta)} a(x,y,\theta) \, \dd \theta,\qquad (x,y) \in \rr {2d},
\end{equation*}
where $a \in\Gamma^m(\rr{2d} \times \rr N)$. 
The phase function $\fy$ is a real quadratic form on $\rr {2d+N}$
which parametrizes the twisted graph Lagrangian 
\begin{equation}\label{eq:twistedgraphlagrangian}
\Lambda_\chi' = \{(x, y, \xi,-\eta) \in T^* \rr {2d}: \  (x,\xi) = \chi (y,\eta) \} \subseteq T^* \rr {2d}
\end{equation}
corresponding to $\chi \in \Sp(d, \ro)$. 
(Cf. \cite[Definitions~3.5 and 4.1]{Cappiello3}.)
We will not use this representation here but merely recall the following result, see also \cite[Theorem~1.3]{CGNR} for a related result where certain modulation spaces are used as amplitudes.

\begin{thm}\label{thm:repFIO} 
{\rm \cite[Theorem~4.15]{Cappiello3}}
If $\chi \in \Sp(d,\ro)$ and $\cK \in \cI^m(\chi)$ then 
there exist $b \in \Gamma^m$ such that for any $\mu \in \Mp(d)$ such that $\chi = \pi(\mu)$
\begin{equation*}
\cK  = b^w(x,D)  \mu = \mu (b \circ \chi)^w(x,D). 
\end{equation*}
Conversely, for any $b \in \Gamma^m$ we have $b^w(x,D) \mu \in \cI^m(\chi)$.
\end{thm}

This means that FIOs in $\cI^m(\chi)$ admit factorization into a pseudodifferential operator and a metaplectic operator corresponding to $\chi$. The factorization is uniquely determined by the order of arrangement. 
In particular $\cI^m(I)$, where $I \in \GL(2d,\ro)$ is the identity matrix, is the space of pseudodifferential operators with Shubin amplitudes of order $m \in \ro$.  
A kernel of the form $K_{1,\varphi}$, i.e. trivial amplitude, 
corresponds to the operator $C_\fy \mu$ where $\chi = \pi(\mu)$ and $C_\fy \in \co \setminus 0$ (cf. \cite{Cappiello3, Hormander2, Leray1}).
A fundamental result for FIOs is the following composition theorem. 

\begin{thm}\label{prop:composition} 
{\rm \cite[Proposition~4.10]{Cappiello3}}
Let $\chi_j \in \Sp(d,\ro)$ and suppose $\cK_j \in \cI^{m_j}(\chi_j)$, for $j=1,2$. 
Then $\cK_1 \cK_2  \in \cI^{m_1+m_2}(\chi_1 \chi_2)$. 
\end{thm}

We state the mapping properties of FIOs with respect to the Shubin--Sobolev spaces $Q^s$ and the Gabor wave front set respectively.

\begin{prop}\label{prop:mapprop} 
{\rm \cite[Proposition~4.16 and Corollary~5.4]{Cappiello3}}
Suppose $\chi \in \Sp(d,\ro)$ and $\cK \in \cI^m(\chi)$. 
Then $\cK: Q^{s}(\rr d) \rightarrow Q^{s-m}(\rr d)$ is continuous for all $s \in \ro$. 
For all $u \in \cS'(\rr d)$ we have
$\WF( \cK u ) \subseteq \chi \WF(u)$. 
\end{prop}

\section{Parametrix and propagator}\label{sec:application}

Consider the initial value Cauchy problem associated with a real homogeneous quadratic form $q \in \Gamma^2$ 
defined by $q(x,\xi) = \la (x,\xi), Q (x,\xi) \ra$ where $(x,\xi) \in \rr {2d}$ and $Q \in \M_{2d \times 2d}(\ro)$ is symmetric, 
and a negative order complex-valued perturbation $p \in \Gamma^{-\delta}$ where $\delta>0$. 
\begin{equation}\label{eq:cp}
\tag{CP}
	\left\{
	\begin{array}{rl}
	\partial_t u(t,x) + i(q^w(x,D)+p^w(x,D)) u (t,x) & = 0, \qquad t > 0, \quad x \in \rr d, \\
	u(0,\cdot) & = u_0\in \cS'(\rr d).  
	\end{array}
	\right.
\end{equation}
\subsection{The free evolution}
We first discuss the solution operator (propagator) in the unperturbed case $p=0$. 
First we treat the propagator as a group on $L^2(\rr d)$, then on $Q^s(\rr d)$. 

Thus we consider $q^w(x,D)$ as an unbounded operator in $L^2(\rr d)$. 
The closure of $-i q^w(x,D)$ equipped with the domain $\cS$ equals its maximal realization, denoted $M_q$ 
\cite[pp.~425--26]{Hormander2}. 
The closure generates a strongly continuous group $\ro \ni t \mapsto e^{-it q^w(x,D)}$ of unitary operators on $L^2$. 
The group gives the unique solution $e^{-it q^w(x,D)} u_0 \in
C([0,\infty), L^2) \cap C^1((0,\infty),L^2)$
for $u_0 \in D( M_q ) \subseteq L^2$, see \cite[Theorem~4.1.3]{Pazy1}. 

The propagator is a time-parametrized group of metaplectic operators, given for $t \in \ro$ by
\begin{equation}\label{eq:metaplecticpropagator}
e^{-it q^w(x,D)} = \mu_t \in \Mp(d), \quad t \in \ro
\end{equation}
(see e.g. \cite{Cappiello3,deGosson2,Folland1,CGNR,PRW1}). 
In fact consider the one-parameter group of symplectic matrices $\chi_t =e^{2t F} \in \Sp(d,\ro)$ where $F = \J Q \in \M_{2d \times 2d}(\ro)$. 
By the unique path lifting theorem (cf. \cite[Corollary~355]{deGosson2}), 
there is a unique continuous lifting of $\ro \ni t \mapsto \chi_t \in \Sp(d,\ro) $ into 
$\ro \ni t \mapsto \mu_t \in \Mp(d)$ such that $\pi(\mu_t) = \chi_t$ for $t \in \ro$ and $\mu_0 = I$. 
By \cite[Corollary~355]{deGosson2} $\mu_t$ satisfies \eqref{eq:metaplecticpropagator}. 

\begin{rem}
Williamson's symplectic diagonalization theorem (see e.g. \cite[Theorem~93]{deGosson2}) implies that 
if $Q  \in \M_{2d \times 2d}(\ro)$ is (strictly) positive definite then there exists a matrix $\chi \in \Sp(d,\ro)$ such that 
\begin{equation*}
\chi^t Q \chi = 
\left(
\begin{array}{ll}
\Lambda & 0 \\
0 & \Lambda
\end{array}
\right)
\end{equation*}
where $\Lambda = \diag(\lambda_1,\cdots, \lambda_d)$
with $\lambda_j, j=1,\ldots, d$ positive numbers such that $\{\pm i \lambda_j\}_{j=1}^d$ are eigenvalues of $F$. 
This gives
\begin{equation*}
(q \circ \chi) (x,\xi) = \la \chi (x,\xi), Q \chi (x,\xi) \ra
= \sum_{j=1}^d \lambda_j (x_j^2 + \xi_j^2)
\end{equation*}
and thus 
\begin{equation}\label{eq:sumharmonicoscillator}
(q \circ \chi)^w(x,D) = \sum_{j=1}^d \lambda_j (x_j^2 - \partial_j^2)
\end{equation}
which is a weighted sum of one-dimensional harmonic oscillators. 
Picking $\mu \in \Mp(d)$ such that $\chi = \pi(\mu)$, and conjugating 
the equation $\partial_t u(t,x) + i q^w(x,D) u (t,x) = 0$ in the $x$ variable 
with $\mu$ using \eqref{eq:metaplecticoperator}, 
leads to the modified Hamiltonian \eqref{eq:sumharmonicoscillator}. 
We will not pursue this direction, however, since our main concern is to
express our results using $\chi_t \in \Sp(d,\ro)$, and the relation between $\chi$ and $\chi_t$
is not transparent.
\end{rem}

Next we fix $s \in \ro$ and consider $q^w(x,D)$ as an unbounded operator in $Q^s(\rr d)$. In this case $\mu_t$ is in general no longer unitary but we still have the following result.

\begin{prop}\label{prop:semigroupQs}
For $s \in \ro$ the group $\ro \ni t \mapsto \mu_t$ is a strongly continuous group of operators on $Q^s(\rr d)$ whose generator is a closed extension of $- i q^w(x,D)$, considered as an unbounded operator in $Q^s(\rr d)$ with domain $\cS(\rr d)$. 
\end{prop}

\begin{proof}
By \cite[Prop.~400]{deGosson2}, $\mu_t$ is for fixed $t \in \ro$ a homeomorphism on $Q^s$. 
First we prove a uniform bound for $\| \mu_t \|_{\cL(Q^s)}$ over $-T \leqs t \leqs T$ where $T>0$.

By \eqref{eq:localizationweyl} we have $A_s = a^w(x,D)$ where $a(z) = \pi^{-d} (e^{-|\cdot|^2} * v_s)(z)$ with $z\in\rr{2d}$.  
This implies $a \in \Gamma^s$
and $a$ is elliptic by \cite[Proposition~2.3]{SW2}. 
Since $\|\mu_t f \|_{L^2} = \| f \|_{L^2}$ for all $t \in \ro$ and $f \in L^2$ we obtain for $u \in Q^s$ using 
\eqref{eq:metaplecticoperator} 
\begin{align*}
\| \mu_t u\|_{Q^s} 
& = \| A_s \mu_t u\|_{L^2} \\
& = \|\mu_t \mu_t^{-1} a^w (x,D) \mu_t u\|_{L^2} \\
& = \| (a \circ \chi_t)^w (x,D) u\|_{L^2} \\
& \leqs \| ( a \circ \chi_t )^w (x,D) A_s^{-1} \|_{\cL(L^2)} \| A_s u\|_{L^2} \\
& = \| ( a \circ \chi_t )^w (x,D) A_s^{-1} \|_{\cL(L^2)}  \| u \|_{Q^s}. 
\end{align*}
Indeed, by \cite[Proposition~1.7.12]{Nicola1} the inverse of $A_s$ exists and $A_s^{-1} = b^w(x,D)$ where $b \in \Gamma^{-s}$.
We have
\begin{equation}\label{eq:symbolestimparam}
\left| \pd \alpha (a \circ \chi_t) (z) \right| \leqs C_\alpha e^{2 |t| \, \| F \| (|s|+ 2 |\alpha|)} \eabs{z}^{s-|\alpha|}, \quad z \in \rr {2d}, \quad \alpha \in \nn {2d}.  
\end{equation}
The set of symbols $a \circ \chi_t \in \Gamma^s$ is thus uniformly bounded over $t \in [-T,T]$. 
Hence $(a \circ \chi_t) \wpr b \in \Gamma^0$ is uniformly bounded over $t \in [-T,T]$. 
By the Calder\'on--Vaillancourt theorem (see e.g. \cite[Theorem~2.73]{Folland1}) $\| ( a \circ \chi_t )^w (x,D) A_s^{-1} \|_{\cL(L^2)} < \infty$ uniformly over $t \in [-T,T]$. 
We have shown 
\begin{equation}\label{eq:unifmutQs}
\sup_{|t| \leqs T} \|\mu_t \|_{\cL(Q^s)} < \infty, \quad s \in \ro. 
\end{equation}

Next let $f \in \cS$ and write as above
\begin{align*}
A_s ( \mu_t - I ) f
& = \mu_t ( a \circ \chi_t )^w (x,D) f - a^w(x,D) f \\
& = \mu_t ( a \circ \chi_t -a)^w (x,D) f + (\mu_t - I) a^w(x,D) f.
\end{align*}
Since $\mu _t$ is unitary on $L^2$
we obtain for $|t| \leqs 1$ 
\begin{align*}
\| ( \mu_t - I ) f \|_{Q^s}
& = \| A_s ( \mu_t - I ) f \|_{L^2} \\
& \leqs \| ( a \circ \chi_t -a)^w (x,D) f \|_{L^2 } + \| (\mu_t - I) a^w(x,D) f \|_{L^2}. 
\end{align*}
We have $a \circ \chi_t -a \to 0$ in $\Gamma^{s+\nu}$ as $t \to 0$ provided $\nu> 0 $. 
To wit, this follows from the proof of \cite[Proposition~18.1.2]{Hormander0} modified from H\"ormander to Shubin symbols. 
This gives  
\begin{equation*}
\lim_{t \to 0} \frac{\| ( a \circ \chi_t -a)^w (x,D) f \|_{L^2 }}{\| f \|_{Q^{s+\nu}}}
= 0. 
\end{equation*}
Combining with the known strong continuity of $\mu_t$ on $L^2$
we have shown 
\begin{equation*}
\lim_{t \to 0} \| (\mu_t -I) f \|_{Q^s} =0, \quad f \in \cS(\rr d). 
\end{equation*}
Finally, combining \eqref{eq:unifmutQs} with the fact that $\cS$ is dense in $Q^s$ we can use \cite[Proposition I.5.3]{Engel1} to conclude that $\mu_t$ is a strongly continuous group on $Q^s$. 

Consider finally the final statement of the proposition: The generator of the group $\mu_t$ acting on $Q^s$
is a closed extension of $- i q^w(x,D)$, considered as an unbounded operator on $Q^s(\rr d)$ with domain $\cS(\rr d)$. 
This claim is a consequence of \cite[Theorem~1.2.4 and Corollary~1.2.5]{Pazy1}. 
\end{proof}

Both sides of the equality \eqref{eq:metaplecticpropagator} may thus be interpreted as a strongly continuous group on $Q^s$. 
The operators $\mu_t$ are not necessarily unitary if $s \neq 0$. 
Due to \eqref{eq:SQs} we may allow $u_0 \in \cS'(\rr d)$. 
In fact for some $s \in \ro$ we then have $u_0 \in Q^{s+2}$. 
The group $\mu_t$ acting on $Q^s$ has a closed generator that is an extension of $- i q^w(x,D)$
considered an unbounded operator in $Q^s(\rr d)$ with domain $\cS(\rr d)$. 
Abusing notation we denote also the generator by $- i q^w(x,D)$.
It follows from \cite[Definition~1.1.1]{Pazy1} and \eqref{eq:shubinpsdocont} that $Q^{s+2} \subseteq D( q^w(x,D) )$. 
Again \cite[Theorem~4.1.3]{Pazy1} implies that $\mu_t u_0$ is the unique solution to \eqref{eq:cp} in 
$C([0,\infty), Q^s) \cap C^1((0,\infty), Q^s)$.
We summarize: 
\begin{prop}
For $s \in \ro$ the equation \eqref{eq:cp} with $p=0$ is solved uniquely by the strongly continuous group of operators $e^{-itq^w(x,D)} = \mu_t$ on $Q^s(\rr d)$, and for each $t \in \ro$ it is an FIO in $\cI^0(\chi_t)$. We have for $u_0 \in Q^{s+2}$ the unique solution 
\begin{equation*}
e^{-itq^w(x,D)}u_0 \in C([0,\infty), Q^s) \cap C^1((0,\infty), Q^s). 
\end{equation*}
\end{prop}

\subsection{Construction of a parametrix to the perturbed equation}

We will now consider \eqref{eq:cp} with a nonzero complex-valued perturbation $p \in \Gamma^{-\delta}$.
As a first step we note that the perturbation operator is bounded $p^w(x,D): Q^s(\rr d) \to Q^{s+\delta} (\rr d)$ 
and compact $p^w(x,D): Q^s(\rr d) \to Q^s (\rr d)$ \cite[Proposition~25.4]{Shubin1}.
Perturbation theory (see e.g. \cite{CGNR}, \cite[Theorems~III.1.3 and III.1.10]{Engel1}) gives the following conclusion. 

Let $s \in \ro$. The solution to \eqref{eq:cp} for $u_0 \in Q^{s+2}(\rr d)$ is $T_t u_0$ where
\begin{equation}\label{eq:propagator1}
T_t = \mu_t C_t, \quad t \geqs 0.  
\end{equation}
Here $C_t$ is a strongly continuous semigroup of operators on $Q^s(\rr d)$ with operator norm estimate
\begin{equation}\label{eq:Ctopnorm}
\| C_t \|_{\cL(Q^s)} \leqs M e^{t \left( \omega + M \| p^w(x,D) \|_{\cL(Q^s)} \right)}, \quad t \geqs 0, 
\end{equation}
where $M \geqs 1$, $\omega \geqs 0$,
and
\begin{equation*}
C_t = \mathrm{id}+\sum_{n=1}^\infty (-i)^n\int_0^t\int_0^{t_1}\cdots \int_0^{t_{n-1}} P_{t_1} \cdots P_{t_n}\,\dd t_n\cdots\dd t_1
\end{equation*}
with convergence in the $\cL(Q^s)$ norm. 
In this formula $P_t=p_t^w(x,D)=(p \circ \chi_t)^w(x,D)$. 
The integrals are Bochner integrals of operator-valued functions. 
The propagator \eqref{eq:propagator1} is a strongly continuous semigroup of operators on $Q^s$. 

By \cite[Corollary~11.2.6 and Lemma~11.3.3]{Grochenig1} the $Q^s$ norms for $s \geqs 0$ are a family of seminorms for $\cS(\rr d)$. 
Thus $C_t: \cS \to \cS$ is continuous. 

We show that $C_t$ is a pseudodifferential operator. First we use results in \cite{CGNR} to prove that $C_t$ has a pseudodifferential operator symbol in a space larger than $\Gamma^0$.
By \cite[Remark~2.18]{Holst1} we have 
\begin{equation*}
\Gamma^{-\delta} \subseteq \Gamma_0^0 = \bigcap_{s \geqs 0} M_{1 \otimes v_s}^{\infty,1} (\rr {2d})
\end{equation*}
where $M_{1 \otimes v_s}^{\infty,1}$ denotes a Sj\"ostrand  modulation space \cite{Grochenig1}
with the weight $v_s (z)$, $z \in \rr {2d}$, and
where $\Gamma_0^0$ denotes the space of smooth symbols whose derivatives are in $L^\infty$. 
From \cite[Theorem~4.1]{CGNR} it follows that 
$C_t = c_t^w(x,D)$ where 
\begin{equation}\label{eq:symbolmodulation}
c_t \in \bigcap_{s \geqs 0} M_{1 \otimes v_s}^{\infty,1} = \Gamma_0^0, \qquad t \geqs 0.  
\end{equation}
By duality $c_t^w(x,D)$ extends uniquely to a continuous operator on $\cS'(\rr d)$. 

The outcome of this argument is that the propagator \eqref{eq:propagator1} is of the form
\begin{equation}\label{eq:propagator2}
T_t = \mu_t c_t^w(x,D), \quad t \geqs 0.  
\end{equation}
If $u_0 \in \cS'(\rr d)$ then $u_0 \in Q^{s+2}$ for some $s \in \ro$.  
Again, by \cite[Theorem~4.1.3]{Pazy1}, $T_t u_0$ is the unique solution to \eqref{eq:cp} in $C([0,\infty), Q^s) \cap C^1((0,\infty), Q^s)$. 

Our objective is to improve \eqref{eq:symbolmodulation} into
\begin{equation}\label{eq:ctsymbol}
c_t \in \Gamma^0, \qquad t \geqs 0,
\end{equation}
which implies that the propagator $T_t$ is an FIO of order zero for all $t \geqs 0$. 
This improvement will prove Theorem \ref{thm:mainresult}. 

The strategy to prove \eqref{eq:ctsymbol} is as follows. We first construct an FIO parametrix $\{ \cK_t \}_{t \geqs 0}$ to the equation \eqref{eq:cp},
that is a family of operators $\{ \cK_t \}_{t \geqs 0}$, where $\cK_t \in \cI^0(\chi_t)$ for $t \geqs 0$, which satisfies 
\begin{equation}\label{eq:CPparametrix}
\left\{
\begin{array}{rl}
\partial_t \cK_t u_0 + i(q^w(x,D)+p^w(x,D)) \cK_t u_0 & = g(t), \qquad t > 0, \\
\cK_0 u_0 & = u_0, \quad u_0 \in \cS'(\rr d), 
\end{array}
\right.
\end{equation}
for a function $g \in C( [0,\infty), \cS(\rr d))$. (The function $g$ will turn out to depend on $u_0$.)
We then prove that $\cK_t - T_t = \cR_t$ is regularizing, which implies that $T_t=\cK_t - \cR_t \in \cI^0(\chi_t)$ is an FIO. 

Thus we start by proving the following result. 

\begin{thm}
\label{thm:parametrix}
The Cauchy problem \eqref{eq:cp} admits an FIO parametrix $\cK_t \in \cI^0( \chi_t )$ for $t \geqs 0$ such that $\cK_0 = I$.
\end{thm}
The proof is carried out in several steps. 

\begin{lem}\label{lem:symbolparambound}
Let $T>0$ and $n \geqs 1$. 
The family of Weyl symbols 
\begin{equation*}
p_{t_1} \wpr \cdots \wpr p_{t_n} \in \Gamma^{-\delta n}, \quad t_j \in [0,T], \quad 1 \leqs j \leqs n, 
\end{equation*}
is uniformly bounded in $\Gamma^{-\delta n}$, and 
$
[0,T]^n \ni (t_1, t_2, \dots,t_n) \mapsto p_{t_1} \wpr  \cdots \wpr p_{t_n} \in \Gamma^{-(\delta - \nu) n}
$
is continuous for any $\nu>0$. 
\end{lem}
\begin{proof}
We have
\begin{equation}\label{eq:symbolestimparam2}
\left| \pd{\alpha}_{z} p_t (z) \right| \leqs C_\alpha e^{2 t \| F \| ( \delta + 2|\alpha|)} \eabs{z}^{-\delta-|\alpha|}, \quad z \in \rr {2d}, \quad \alpha \in \nn {2d}, 
\end{equation}
which proves that $p_t \in \Gamma^{-\delta}$ uniformly over $t \in [0,T]$. 
The estimates \eqref{eq:symbolestimparam2} combined with the continuity of $t \mapsto \pd{\alpha}_{z} p_t (z)$ also show that 
\begin{equation}\label{eq:symbolcont}
p_t \in C( [0,T],  \Gamma^{-\delta + \nu})
\end{equation}
if $\nu>0$. 
The result is thus a consequence of the continuity of the Weyl product on the spaces $\Gamma^m$ (see e.g. \cite[Theorem~1.2.16]{Nicola1}). 
\end{proof}

Fix $T>0$. For $t \in [0,T]$ we set $b_{t,0}=1$, and for $n \geqs 1$
\begin{equation}\label{eq:btndef}
b_{t,n} = (-i)^n\int_0^t\int_0^{t_1}\cdots \int_0^{t_{n-1}} p_{t_1} \wpr  \cdots \wpr p_{t_n} \,\dd t_n \cdots\dd t_1.
\end{equation}
In particular, $b_{0,n} = 0$ for $n \geqs 1$. 
By Lemma \ref{lem:symbolparambound}, $b_{t,n}$ makes sense as an integral 
and $b_{t,n} \in \Gamma^{- \delta n}$.  
From \eqref{eq:wignerweyl} it follows that integration commutes with the Weyl product so that 
\begin{equation*}
b_{t,n}^w(x,D) = (-i)^n\int_0^t\int_0^{t_1}\cdots \int_0^{t_{n-1}} P_{t_1} \cdots P_{t_n}\,\dd t_n\cdots\dd t_1, \quad n \geqs 0.
\end{equation*}

Using \eqref{eq:symbolestimparam2}, the recursion
\begin{equation}\label{eq:btrecursion}
b_{t,n} = -i \int_0^t p_{t_1} \wpr  b_{t_1,n-1} \dd t_1, \quad n \geqs 1, 
\end{equation}
and induction, one shows 
\begin{equation}\label{eq:btncont}
t \mapsto b_{t,n} \in C( [0,T], \Gamma^{- \delta n}), \qquad n \geqs 0.
\end{equation}
Hence
\begin{equation}\label{eq:uniformboundsymbol}
\left| \pd{\alpha}_z b_{t,n} (z) \right| \leqs C_\alpha  \eabs{z}^{-\delta n- |\alpha|}, \quad \alpha \in \nn {2d}, \quad 0 \leqs t \leqs T.  
\end{equation}
We also have for $n \geqs 1$
\begin{equation}\label{eq:uniformboundderivative}
\left| \pd{\alpha}_z \partial_t b_{t,n} (z) \right| \leqs C_\alpha  \eabs{z}^{-\delta n - |\alpha|}, \quad \alpha \in \nn {2d}, \quad 0 \leqs t \leqs T.  
\end{equation}

In fact \eqref{eq:btrecursion} gives for $n \geqs 1$ and $t>0$
\begin{equation}\label{eq:btident}
\partial_t b_{t,n} 
= -i \, p_t \wpr b_{t,n-1}. 
\end{equation}
Hence \eqref{eq:symbolcont} and \eqref{eq:btncont}
show that 
\begin{equation}\label{eq:dtbtncont}
t \mapsto \partial_t b_{t,n} \in C( [0,T], \Gamma^{- \delta n + \nu})
\end{equation}
provided $\nu>0$. 
(Note that the continuity extends to include the end points of the interval $[0,T]$.)
Lemma \ref{lem:symbolparambound} and \eqref{eq:uniformboundsymbol} imply that
$\partial_t b_{t,n} \in \Gamma^{- \delta n}$
uniformly over $t \in [0,T]$. This proves \eqref{eq:uniformboundderivative}. 
By differentiating \eqref{eq:btident} and using 
\begin{equation*}
\partial_t (p_t (z)) 
= \la (\nabla p) (\chi_t z), 2 F \chi_t z \ra \in \Gamma^{-\delta}
\end{equation*}
we also obtain 
\begin{equation}\label{eq:uniform2derivative}
\partial^2_t b_{t,n} \in \Gamma^{-\delta n}
\end{equation} 
uniformly over $t \in [0,T]$ for all $n \geqs 1$. 

We need the following technical result to sum the $b_{t,n}$ asymptotically.

\begin{lem} \label{lem:sviluppo}
Let $T>0$ and $a_{t,n} \in C([0,T], \Gamma^{m_n})$ where $(m_n)_{n \in \mathbb{N}}$ is a decreasing sequence tending to $-\infty$ as $n \to \infty$. Assume that $\partial_t a_{t,n} \in \Gamma^{m_n}$ uniformly over $t \in [0,T]$ for every $n \in \mathbb{N}.$ Then there exists a symbol function $t \mapsto a_t \in C([0,T], \Gamma^{m_0}) $ such that for every $N \in \mathbb{N} \setminus 0$ we have
\begin{equation*}
a_t - \sum_{n=0}^N a_{t,n} \in C([0,T], \Gamma^{m_{N+1}}).
\end{equation*}
We write
$\displaystyle a_{t} \sim \sum_{n=0}^\infty a_{t,n}. $
\end{lem}

The proof is a variant of the proof of \cite[Proposition 3.5]{Shubin1} and other similar results for symbols depending on a parameter. The essential modification of the standard proof to obtain continuity is to use 
\begin{equation*}
\left| \pdd z \alpha \left( a_{t+s,n}(z)-a_{t,n}(z) \right) \right| 
\leqs |s|  \sup_{0 \leqs \theta \leqs 1}  \left| \pdd z \alpha \partial_t a_{t+\theta s,n}(z) \right|. 
\end{equation*}
We omit further details except for the statement that 
the symbol $a_t$ is constructed as
\begin{equation}\label{eq:btconstruction}
a_t = \sum_{n=0}^\infty \psi_n a_{t,n}
\end{equation}
where $\psi_0(z)=1$, and for $n \geqs 1$, $\psi_n(z)=\psi(z/r_n)$ where $\psi \in C^\infty(\rr {2d})$, $\psi(z)=0$ for $|z| \leqs 1$, $\psi(z)=1$ for $|z| \geqs 2$, 
and $(r_n)_{n \geqs 1}$ a sufficiently rapidly increasing sequence of positive reals. 
This shows that we can extend
Lemma \ref{lem:sviluppo} to take into account also higher order derivatives with respect to $t$ by modifying the constants $r_n$. 

Applying Lemma \ref{lem:sviluppo} to the sequences $(b_{t,n})_{n \geqs 0}$ defined in \eqref{eq:btndef} and $(\partial_t b_{t,n})_{n \geqs 1}$ simultaneously, and using
\eqref{eq:btncont}, \eqref{eq:uniformboundsymbol}, \eqref{eq:uniformboundderivative}, 
\eqref{eq:dtbtncont} and \eqref{eq:uniform2derivative} we obtain 
$b_t \in C([0,T], \Gamma^0)\cap C^1([0,T], \Gamma^{-\delta})$
such that 
\begin{equation*}
b_t \sim \sum_{n=0}^\infty b_{t,n}, \quad 
\partial_t b_t \sim \sum_{n=1}^\infty \partial_t b_{t,n}. 
\end{equation*}
Note that $b_0^w(x,D) = I$. 
 
\begin{lem}\label{lem:rtbtcont}
For $T>0$ we have
\begin{equation*}
r_t = \partial_t b_t + i p_t \wpr b_t \in C([0,T], \cS(\rr {2d})).
\end{equation*}
\end{lem}

\begin{proof}
Let $N \in \no \setminus 0$.  
By \eqref{eq:btident} and Lemma \ref{lem:sviluppo} we have 
\begin{equation*}
\partial_t b_t + i \sum_{n=1}^N p_t  \wpr b_{t,n-1} = \partial_t b_t -\sum_{n=1}^N \partial_t b_{t,n} \in C( [0,T], \Gamma^{-\delta(N+1)}). 
\end{equation*}
On the other hand we also observe that 
\begin{equation*}
i p_t  \wpr b_t - i \sum_{n=1}^N p_t  \wpr b_{t,n-1} 
= i p_t \wpr \left( b_t - \sum_{n=0}^{N-1} b_{t,n} \right)  \in C([0,T], \Gamma^{-\delta (N+1)+\nu})
\end{equation*}
again using Lemma \ref{lem:sviluppo} and \eqref{eq:symbolcont}. 
Since $N>0$ is arbitrary the claim follows. 
\end{proof}

\noindent \emph{Proof of Theorem \ref{thm:parametrix}.}
Weyl quantization of Lemma \ref{lem:rtbtcont} gives
\begin{equation*}
\partial_t b^w_t(x,D) = - i p_t^w(x,D) \, b_t^w(x,D) + r_t^w(x,D), \qquad t \geqs 0. 
\end{equation*}
For $u_0 \in \cS'(\rr d)$ we define $\cK_0 u_0 = u_0$ and for $t > 0$
\begin{equation*}
v_t = \cK_t u_0 = \mu_t b_t^w(x,D)u_0 \in \cS'(\rr d). 
\end{equation*}
For $t>0$ we have
\begin{equation}\label{eq:dtvt}
\begin{aligned}
i\partial_t v_t&=i\partial_t \left(\mu_t b_t^w(x,D)u_0\right) \\
&=q^w(x,D) v_t + \mu_t \left( p_t^w(x,D) \, b_t^w(x,D) + i r_t^w(x,D) \right)  u_0 \\
&=q^w(x,D) v_t + p^w(x,D) v_t + i \mu_t r_t^w(x,D) u_0.
\end{aligned}
\end{equation}
If we set 
\begin{equation*}
g(t) = \mu_t r_t^w(x,D) u_0
\end{equation*}
it remains to show $g \in C([0,\infty),\cS(\rr d))$. 
Writing 
\begin{equation*}
g(t+s) - g(t) = \mu_{t+s} \left(  r_{t+s}^w(x,D)  - r_t^w(x,D) \right) u_0
+  \mu_t \left(  \mu_s - I \right) r_t^w(x,D) u_0
\end{equation*}
the latter claim is a consequence of Proposition \ref{prop:semigroupQs}, Lemma \ref{lem:rtbtcont}, \eqref{eq:shubinpsdocont}, 
and the fact that by \eqref{eq:SQs} $u_0 \in Q^m$ for some $m \in \ro$. 
We have thus shown that $\cK_t$ is a parametrix to \eqref{eq:cp}. \hfill $\Box$

\subsection{The propagator to the perturbed equation as an FIO}
In order to show \eqref{eq:ctsymbol} it remains to connect the parametrix with the solution operator \eqref{eq:propagator2}
(cf. the proof of \cite[Proposition 3.1.1]{Helffer1}). 
Let $u_0 \in \cS'(\rr d )$. 
Then for some $s \in \ro$ we have $u_0 \in Q^{s+2}(\rr d) \subseteq D\big( q^w(x,D) + p^w(x,D)\big)$ 
where $q^w(x,D) + p^w(x,D)$ is considered an unbounded operator in $Q^s(\rr d)$. 
\begin{lem}\label{lem:Qscontinuity}
If $T>0$ then
\begin{equation}\label{eq:mutrtcontinuity2}
t \mapsto \mu_t r_t^w(x,D) u_0 \in C([0,T], Q^s), 
\end{equation}
\begin{equation}\label{eq:Ktcontinuity2}
t \mapsto v_t \in C([0,T], Q^s), \qquad \mbox{and}
\end{equation}
\begin{equation}\label{eq:Ktcontinuity3}
t \mapsto \partial_t v_t \in C((0,T), Q^s).
\end{equation}
\end{lem}
\begin{proof}
Property \eqref{eq:mutrtcontinuity2} is a consequence of Proposition \ref{prop:semigroupQs}, 
Lemma \ref{lem:rtbtcont} and \eqref{eq:shubinpsdocont}. 
Lemma \ref{lem:sviluppo} gives $b_t \in C([0,T], \Gamma^0)$. 
Hence property \eqref{eq:Ktcontinuity2} follows from Proposition \ref{prop:semigroupQs}
and  \eqref{eq:shubinpsdocont}. 

Finally we show \eqref{eq:Ktcontinuity3}. 
From \eqref{eq:dtvt} we obtain for $t>0$
\begin{equation*}
\partial_t v_t
= - i q^w (x,D) \mu_t b_t^w (x,D) u_0 - i p^w (x,D) \mu_t b_t^w (x,D) u_0 + \mu_t r_t^w (x,D) u_0. 
\end{equation*}
In order to prove \eqref{eq:Ktcontinuity3} we must show
\begin{align}
& t \mapsto p^w (x,D) \mu_t b_t^w (x,D) u_0 \in C((0,T), Q^s), \label{eq:subcase2} \\
& t \mapsto q^w (x,D) \mu_t b_t^w (x,D) u_0 \in C((0,T), Q^s). \label{eq:subcase3}
\end{align}
by virtue of \eqref{eq:mutrtcontinuity2}. 
Claim \eqref{eq:subcase2} is a consequence of \eqref{eq:Ktcontinuity2}, $p \in \Gamma^{-\delta}$ and \eqref{eq:shubinpsdocont}. 
Finally claim \eqref{eq:subcase3} is a consequence of $q \in \Gamma^2$, \eqref{eq:shubinpsdocont} and Proposition \ref{prop:semigroupQs}.
\end{proof}

From Lemma \ref{lem:Qscontinuity} we may conclude that 
\begin{equation*}
t \mapsto \cK_t  u_0 \in C( [0,T], Q^s) \cap C^1( (0,T), Q^s). 
\end{equation*}
From Proposition \ref{prop:semigroupQs} combined with $b_t \in C([0,T], \Gamma^0)$
and \eqref{eq:shubinpsdocont} we obtain 
$\cK_t  u_0 \in Q^{s+2} \subseteq D\big( q^w(x,D) + p^w(x,D)\big)$ for $t > 0$. 
Since $[0,T] \ni t \mapsto \cK_t  u_0$ solves \eqref{eq:CPparametrix} with $u_0 \in Q^{s+2}(\rr d)$, 
it is a classical solution according to \cite[Definition~4.2.1]{Pazy1}. 

Assembling the knowledge from \eqref{eq:propagator2} and Theorem  \ref{thm:parametrix} gives the following conclusion. 
The map $t \mapsto (\cK_t - T_t) u_0$ solves the equation
\begin{equation*}
\left\{
\begin{array}{rl}
\partial_t (\cK_t - T_t) u_0 + i(q^w(x,D)+p^w(x,D)) (\cK_t -T_t) u_0 & = \mu_t r_t^w(x,D)u_0, \quad t > 0, \\
(\cK_0 - T_0) u_0 & = 0
\end{array}
\right.
\end{equation*}
and 
\begin{equation*}
t \mapsto (\cK_t - T_t)  u_0 \in C( [0,T], Q^s) \cap C^1( (0,T), Q^s). 
\end{equation*}

Combining \eqref{eq:mutrtcontinuity2} in Lemma \ref{lem:Qscontinuity} with \cite[Corollary~4.2.2]{Pazy1} gives, 
invoking \eqref{eq:propagator2}, 
\begin{align*}
\cR_t u_0 & : = (\cK_t - T_t )u_0 = \int_0^t T_{t-s} \mu_s r_s^w(x,D) u_0 \, \dd s \\
& =  \mu_t \int_0^t \mu_{-s} c_{t-s}^w(x,D) \mu_s r_s^w(x,D) u_0 \, \dd s. 
\end{align*}
\begin{lem}\label{lem:regularizer}
For $t>0$ the kernel of the operator $\mu_{-t} \cR_t$ belongs to $\cS(\rr {2d})$. 
\end{lem}
\begin{proof}
By \cite[Eq.~(50.17) p.~525 and Theorem~51.6]{Treves1} the conclusion follows if we can show that the operator
\begin{equation*}
\int_0^t \mu_{-s} c_{t-s}^w(x,D) \mu_s r_s^w(x,D) \, \dd s
\end{equation*}
is continuous $\cS'(\rr d) \to \cS(\rr d)$ when $\cS'(\rr d)$ is equipped with its strong topology \cite[Section~V.7]{Reed1}. 
By Proposition \ref{prop:semigroupQs} and \eqref{eq:Ctopnorm}, intersected over $s \geqs 0$, 
it suffices to show that $r_t^w(x,D): \cS'(\rr d) \to \cS(\rr d)$ is continuous uniformly over $t \in [0,T]$, 
when $\cS'(\rr d)$ is equipped with the strong topology. 

Let $K_t \in \cS(\rr {2d})$ denote the kernel of $r_t^w(x,D)$. 
According to \eqref{eq:weylsymbolkernel} $K_t$ and $r_t$ are related by the composition of a linear change of variables and partial inverse Fourier transform. 
These are continuous operators on $\cS(\rr {2d})$ and thus Lemma \ref{lem:rtbtcont} implies
$K_t \in C([0,T], \cS(\rr {2d}))$. 
 
We use the seminorms on $g \in \cS(\rr d)$ 
\begin{equation*}
\| g \|_{\alpha,\beta} = \sup_{x \in \rr d} \left| x^\alpha \pd \beta g(x) \right|, \qquad \alpha,\beta \in \nn d. 
\end{equation*}
Let $u \in \cS'(\rr d)$ and $\alpha,\beta \in \nn d$. 
We have
\begin{equation}\label{eq:seminormest}
\| r_t^w(x,D) u \|_{\alpha,\beta}
= \sup_{x \in \rr d} \left| \la x^\alpha \pdd x \beta K_t (x,\cdot), u \ra \right| = \sup_{g \in B} \left| \la g, u \ra \right|, 
\end{equation}
where 
\begin{equation*}
B = \{  x^\alpha \pdd x \beta K_t (x,\cdot) \in \cS(\rr d), \ x \in \rr d \} \subseteq \cS(\rr d). 
\end{equation*}
Let $\gamma, \kappa \in \nn d$ be arbitrary. 
We have
\begin{equation*}
\sup_{g \in B} \| g \|_{\gamma,\kappa}
= \sup_{x,y \in \rr d} \left| x^\alpha y^\gamma \pdd x \beta \pdd y \kappa K_t (x,y ) \right| 
= \| K_t \|_{(\alpha, \gamma), (\beta, \kappa)} < \infty
\end{equation*}
where the latter seminorm bound is uniform over $t \in [0,T]$. 
Thus $B \subseteq \cS(\rr d)$ is a bounded set which implies that $u \mapsto \sup_{g \in B} \left| \la g, u \ra \right|$
is a seminorm on $\cS'(\rr d)$ endowed with the strong topology. 

Since $\alpha, \beta \in \nn d$ are arbitrary, \eqref{eq:seminormest} combined with 
\cite[Theorem~V.2]{Reed1} thus prove the claim that 
$r_t^w(x,D): \cS'(\rr d) \to \cS(\rr d)$ is continuous uniformly over $t \in [0,T]$, 
when $\cS'(\rr d)$ is equipped with the strong topology. 
\end{proof}

Combining Lemma \ref{lem:regularizer} with the fact that an operator has kernel in $\cS(\rr {2d})$ if and only if its Weyl symbol belongs to the same space, we obtain
\begin{equation*}
\cR_t = \mu_t a_t^w(x,D)
\end{equation*}
where $a_t \in \cS(\rr {2d})$ and $t \geqs 0$. 
This gives finally 
\begin{equation*}
T_t = \cK_t - \cR_t = \mu_t \left( b_t^w(x,D) - a_t^w(x,D) \right), \qquad t \geqs 0, 
\end{equation*}
which, in view of \eqref{eq:propagator2}, 
implies that $c_t = b_t - a_t \in \Gamma^0$ 
which is the sought improvement to \eqref{eq:symbolmodulation}.
Thus \eqref{eq:ctsymbol} has at long last been proved. 
This means that we have finally obtained our main result Theorem \ref{thm:mainresult}, which can be alternatively phrased:  
\emph{If $\delta >0$ and $p \in \Gamma^{-\delta}$ then the Cauchy problem \eqref{eq:cp} has an FIO propagator $T_t \in \cI^0( \chi_t)$ for $t \geqs 0$.} 

\section{Singularities of the solutions}\label{sec:sing}

In this section we will discuss propagation of singularities under equation \eqref{eq:cp}. 
Theorem \ref{thm:mainresult} and Proposition \ref{prop:mapprop} give the following result. 

\begin{prop}\label{prop:FIOWFpropagator}
Let $T_t \in \cI^0( e^{2 t F} )$ be the propagator to Cauchy problem \eqref{eq:cp}. 
If $u_0 \in \cS'(\rr d)$ then for $t \geqs 0$ 
\begin{equation*}
\WF( T_t u_0) \subseteq e^{2tF} \WF(u_0).
\end{equation*}
\end{prop}

\begin{rem}
This result can also be seen as a consequence of \cite[Proposition~2.2]{Hormander2}, and
\cite[Theorem~5.1]{Rodino1} combined with \eqref{eq:symbolmodulation}.
\end{rem}

A more refined concept of singularities are (Shubin) $\Gamma$-Lagrangian distributions \cite{Cappiello3} 
which we now explain. 
A Lagrangian linear subspace $\Lambda \subseteq T^*\rr d$ is a space of dimension $d$ on which the restriction of $\sigma$ vanishes. 
If $\Lambda$ is Lagrangian then so is $\chi \Lambda$ for each $\chi \in \Sp(d,\ro)$.
An example of a Lagrangian is $\Lambda_0=\rr{d}\times\{0\}$, and 
all other Lagrangians may be obtained by application of elements $\chi\in \Sp(d,\ro)$ to $\Lambda_0$ \cite{Cappiello3}.

A Lagrangian $\Lambda \subseteq T^*\rr d$ may be parametrized in the form  
\begin{equation}\label{eq:Lagrangian}
\Lambda = \{ (X, AX + Z) \in T^* \rr d, \ X \in Y, \ Z \in Y^\perp \} \subseteq T^*\rr d
\end{equation}
where $Y \subseteq \rr d$ is a linear subspace and $A \in \M_{d \times d}( \ro )$ is a symmetric matrix that leaves $Y$ invariant, see \cite{PRW1}. 
It then automatically leaves $Y^\perp$ invariant so can be written
$A = A_Y + A_{Y^\perp}$
where $A_Y = \pi_Y A \pi_Y$ and $A_{Y^\perp} = \pi_{Y^\perp} A \pi_{Y^\perp}$. 

Note that the subspace $Y \subseteq \rr d$ is uniquely determined by $\Lambda$, but the matrix $A$ is not. 
In fact $A_Y$ is uniquely determined, but $A_{Y^\perp}$ can be any matrix such that $Y \subseteq \Ker A_{Y^\perp}$ and $A_{Y^\perp}$ leaves $Y^\perp$ invariant.
  
The topological space of $\Gamma$-Lagrangian distributions of order $m \in \ro$
with respect to a Lagrangian $\Lambda \subseteq T^*\rr d$ 
is a subspace denoted $I_\Gamma^m(\rr{d},\Lambda) \subseteq \cS'(\rr d)$ \cite[Definition~6.3]{Cappiello3}. 
By \cite[Corollary~6.12]{Cappiello3} this space can be defined as follows. 

\begin{defn}
A distribution $u\in\cS'(\rr d)$ satisfies $u\in I_\Gamma^m(\rr{d},\Lambda)$ if there exist $\chi \in \Sp(d,\ro)$ that maps $\chi: \rr d \times\{0\} \to \Lambda$ isomorphically, and $a \in \Gamma^m(\rr d)$ such that $u= \mu a$ where $\mu \in \Mp(d)$ satisfies $\pi(\mu) = \chi$. 
\end{defn}

By \cite[Proposition~6.7]{Cappiello3} we have $\WF(u) \subseteq \Lambda$ when $u \in I_\Gamma^m(\rr{d},\Lambda)$. 
Kernels of FIOs are $\Gamma$-Lagrangian distributions associated with the twisted graph Lagrangian \eqref{eq:twistedgraphlagrangian} 
of $\chi$ in $T^* \rr {2d}$ \cite[Theorem~7.2]{Cappiello3}. 
We have the following result on the action of FIOs on $\Gamma$-Lagrangian distributions. 

\begin{thm}
\label{thm:FIOonLag}
{\rm \cite[Theorem~6.11]{Cappiello3}}
Suppose $\chi \in \Sp(d, \ro )$, $\cK \in \cI^{m'}(\chi)$ and let $\Lambda \subseteq T^* \rr d$ be a Lagrangian. 
Then
\begin{equation*}
\cK: I^m_\Gamma(\rr d,\Lambda) \to  I^{m+m'}_\Gamma (\rr d,\chi \Lambda)
\end{equation*}
is continuous. 
\end{thm}

As a consequence we obtain the following result
which can be seen as a refinement of Proposition \ref{prop:FIOWFpropagator}.

\begin{thm}
\label{thm:Lagsol}
Suppose $\delta >0$ and $p \in \Gamma^{-\delta}$. 
Let $T_t \in \cI^0( e^{2 t F} )$ be the propagator to Cauchy problem \eqref{eq:cp},
let $\Lambda \subseteq T^* \rr d$ be a Lagrangian
and let $u_0 \in I^m_\Gamma(\rr d, \Lambda)$. 
Then for all $t \geqs 0$
\begin{equation*}
T_t u_0 \in I^m_\Gamma (\rr d, e^{2 t F} \Lambda). 
\end{equation*}
\end{thm}

\subsection{Phase space estimates on the solutions}

In this final section we derive phase space estimates for the propagator and solutions to \eqref{eq:cp}. 
The estimates for the propagator will be relative to the underlying twisted graph Lagrangian \eqref{eq:twistedgraphlagrangian} of a symplectic matrix $\chi \in \Sp(d,\ro)$. 
For this purpose we adapt the integral transform $\cT_g$ (cf. Definition \ref{def:FBItransform}) to 
a matrix $\chi \in \Sp(d,\ro)$ and to a Lagrangian $\Lambda \subseteq T^*\rr d$ respectively, 
see \cite[Section~5 and Proposition~6.14]{Cappiello3}. 

\begin{defn}
We define the following phase factor adjusted versions of the transform $\cT_g$ where $g \in \cS(\rr d) \setminus 0$. 
\begin{enumerate}
\item If $\chi \in \Sp(d,\ro)$ and $u \in \cS'(\rr {2d})$ then
\begin{equation*}
\cT_{g \otimes g}^\chi u (z,\zeta)  = e^{-\frac{i}{2} \left( \la z, \zeta \ra + \sigma(\chi(z_2,-\zeta_2), (z_1,\zeta_1 ) ) \right) } \cT_{g \otimes g}u (z,\zeta), \quad (z,\zeta) \in T^* \rr {2d}.
\end{equation*}
\item If $\Lambda \subseteq T^*\rr d$ is a Lagrangian parametrized by $Y \subseteq \rr d$ and $A \in \M_{d \times d}(\ro)$ as in \eqref{eq:Lagrangian} and $u \in \cS'(\rr{d})$ then 
\begin{equation*}
\cT_g^\Lambda u(x,\xi) = e^{-i \left( \la \pi_{Y^\perp} x, \xi \ra + \frac{1}{2} \la x, A x\ra \right)} \cT_g u (x, \xi), \quad (x,\xi) \in T^* \rr d.
\end{equation*}
\end{enumerate}
\end{defn}

We have the following characterization of the kernels of FIOs (cf. \cite[Theorem~5.2]{Cappiello3} and \cite{Tataru}). 

\begin{prop}\label{thm:FIOkernelchar}
Let $K \in \cS'(\rr {2d})$, $\chi \in \Sp(d,\ro)$ and $g \in \cS(\rr {d}) \setminus 0$. Then $K$ is the kernel of an FIO in $\cI^m( \chi)$ 
if and only if the estimates 
\begin{equation*}
\begin{aligned}
| L_1 \cdots L_k 
\cT_{g \otimes g}^\chi K (z, \zeta)| 
& \lesssim ( 1 + \dist((z,\zeta), \Lambda_{-\chi}') )^{m-k} \, ( 1+\dist((z,\zeta), \Lambda_\chi') )^{-N}, 
\end{aligned}
\end{equation*}
hold for all $k,N \in \no$, where $( z,\zeta) \in T^* \rr {2d}$ and $L_j = \langle a_j, \nabla_{z,\zeta} \rangle$
with $a_j \in \Lambda_\chi'$ for $j=1,2,\dots,k$. 
\end{prop}

We also have the following phase space characterization of Lagrangian distributions \cite[Proposition~6.14]{Cappiello3}.

\begin{prop}\label{prop:Lagrangianchar}
Let $\Lambda \subseteq T^* \rr d$ be a Lagrangian 
parametrized by $Y \subseteq \rr d$ and $A \in \M_{d \times d}(\ro)$ as in \eqref{eq:Lagrangian}, 
and let $V \subseteq T^* \rr d$ be a subspace transversal to $\Lambda$. 
A distribution $u \in \cS'(\rr d)$ satisfies $u \in I_\Gamma^m(\rr d,\Lambda)$ if and only if
for any $g \in \cS(\rr d)\setminus 0$ and for any $k,N \in \mathbb{N}$ we have
\begin{equation*}
\begin{aligned}
\left| L_1 \cdots L_k \cT^\Lambda_g u (x,\xi) \right |
& \lesssim \left( 1 + \dist((x,\xi),V) \right)^{m-k} \left( 1 + \dist((x,\xi),\Lambda) \right)^{-N},
\end{aligned}
\end{equation*}
with $(x,\xi) \in T^* \rr d$, where 
$L_j =\la b_j, \nabla_{x,\xi} \ra$ are first order differential operators with $b_j \in \Lambda$, $j=1,\dots,k$. 
\end{prop}

Applying Propositions \ref{thm:FIOkernelchar} and \ref{prop:Lagrangianchar} to Theorems \ref{thm:mainresult} and \ref{thm:Lagsol} respectively, we obtain the following phase space estimates for the propagator $T_t$ and the solution to \eqref{eq:cp}, see also \cite{CGNR,Tataru} for related results in different symbol classes. Note that our regularity assumptions allow for a precise estimate also of the derivatives of the propagator and solutions.

\begin{thm}
\label{thm:phasespaceest}
Let $\delta >0$, $p \in \Gamma^{-\delta}$, $g \in \cS(\rr {d}) \setminus 0$, $\chi_t = e^{2 t F}$, and denote 
by $T_t \in \cI^0( \chi_t )$, $t \geqs 0$, the propagator to the Cauchy problem \eqref{eq:cp}. 

\begin{enumerate}

\item The kernel $K_t$ of the propagator $T_t$ satisfies the estimates for $t \geqs 0$
\begin{equation*}
\begin{aligned}
| L_1 \cdots L_k 
\cT_{g \otimes g}^{\chi_t} K_t (z, \zeta)| 
& \lesssim ( 1 + \dist((z,\zeta), \Lambda_{-\chi_t}') )^{-k} \, ( 1+\dist((z,\zeta), \Lambda_{\chi_t}') )^{-N}, \\
& \qquad  \qquad ( z,\zeta) \in T^* \rr {2d}, 
\end{aligned}
\end{equation*}
for all $k,N \in \no$, where 
$L_j = \langle a_j, \nabla_{z,\zeta} \rangle$
with $a_j \in \Lambda_{\chi_t}'$ for $j=1,2,\dots,k$. 

\item Suppose $\Lambda \subseteq T^* \rr d$ is a Lagrangian and set $\Lambda_t=\chi_t\Lambda$. 
If $u_0\in I_\Gamma^m(\rr d, \Lambda)$ then the solution $T_t u_0$ to \eqref{eq:cp} satisfies for $t \geqs 0,$ $k,N \in \no$ and $ (x,\xi) \in T^* \rr d$
\begin{equation*}
\left| L_1 \cdots L_k \cT^{\Lambda_t}_g (T_t u_0) (x,\xi) \right |
\lesssim \left( 1 + \dist((x,\xi),V_t) \right)^{m-k} \left( 1 + \dist((x,\xi),\Lambda_t) \right)^{-N}, 
\end{equation*}
where $L_j =\la b_j, \nabla_{x,\xi} \ra$ are first order differential operators with $b_j \in \Lambda_t$, $j=1,\dots,k$ and $V_t$ is a subspace transversal to $\Lambda_t$.
\end{enumerate}
\end{thm}

\bibliographystyle{amsplain}

\begin{thebibliography}{10}

\bibitem{Asada1}
K.~Asada and D.~Fujiwara, \textit{On some oscillatory integral transformations in $L^2(\rr n)$}, Japan J. Math. \textbf{4} (2) (1978), 299--361.
%
\bibitem{BBR}
P.~Boggiatto, E.~Buzano and L.~Rodino, \textit{Global Hypoellipticity and Spectral Theory}, Math. Res. {\bf 92}, Akademie Verlag, Berlin, 1996. 
%
\bibitem{CGR}
M.~Cappiello, T.~Gramchev and L.~Rodino, \textit{Super-exponential decay and holomorphic extensions for semilinear equations with polynomial coefficients}, J. Funct. Anal. \textbf{237} (2006), 634--654. 
%
\bibitem{CN}
M.~Cappiello and F.~Nicola, \textit{Regularity and decay of solutions of nonlinear harmonic oscillators}, Adv. Math. \textbf{229} (2012), 1266--1299.
%
\bibitem{CRT}
M.~Cappiello, L.~Rodino and J.~Toft, \textit{On the inverse to the harmonic oscillator}, Comm. Partial Differential Equations \textbf{40} (6) (2015), 1096--1118.

\bibitem{Cappiello2}
M.~Cappiello, R.~Schulz and P.~Wahlberg, \textit{Conormal distributions in the Shubin calculus of pseudodifferential operators}, 
J. Math. Phys. \textbf{59} (2) (2018), 021502,  18 pp. 

\bibitem{Cappiello3}
M.~Cappiello, R.~Schulz and P.~Wahlberg, \textit{Lagrangian distributions and Fourier integral operators with quadratic phase functions and Shubin amplitudes}, arXiv:1802.04729, 2018.

\bibitem{CGNR}
E.~Cordero, K.~Gr\"ochenig, F.~Nicola and L. Rodino, \textit{Generalized metaplectic operators and the Schr\"odinger equation with a potential in the Sj\"ostrand class},  J. Math. Phys. \textbf{55} (8) (2014), 081506, 17 pp. 

\bibitem{CNR}
E.~Cordero, F.~Nicola and L.~Rodino, \textit{Integral representation for the class of generalized metaplectic operators}, J. Fourier Anal. Appl. \textbf{21} (2015), 694--714.

\bibitem{Folland1}
G.~B.~Folland, \textit{Harmonic Analysis in Phase Space}, Princeton University Press, 1989.

\bibitem{Engel1}
K.-J.~Engel and R.~Nagel, \textit{One-Parameter Semigroups for Linear Evolution Equations}, Springer GTM \textbf{194}, 2000.


\bibitem{deGosson2}
M.~A.~de~Gosson, \textit{Symplectic Methods in Harmonic Analysis and in Mathematical Physics}, Pseudo-Differential Operators, Theory and Applications \textbf{7}, Birkh\"auser/Springer, Basel, 2011.

\bibitem{Grochenig1}
K.~Gr\" ochenig, \textit{Foundations of Time-Frequency Analysis}, Birkh\" auser, Boston, 2001.

\bibitem{Helffer1}
B.~Helffer, \textit{Th\'eorie spectrale pour des op\'erateurs globalement elliptiques}, Ast\'erisque \textbf{112}, 1984.

\bibitem{HR1}
B.~Helffer and D.~Robert, \textit{Comportement asymptotique pr\'ecis\'{e} du spectre d'op\'{e}rateurs
globalement elliptiques dans $\rr n$}, S\'{e}minaire \'{e}quations aux d\'{e}riv\'{e}es partielles (Polytechnique) (1980-1981), exp. no 2, pp. 1--22.

\bibitem{Holst1}
A.~Holst, J.~Toft and P.~Wahlberg, \textit{Weyl product algebras and modulation spaces}, 
J. Funct. Anal. \textbf{251} (2007), 463--491. 

\bibitem{Hormander0}
L.~H\"ormander,
\textit{The Analysis of Linear Partial Differential Operators}, Vol. I, III, IV
Springer, Berlin, 1990.

\bibitem{Hormander1}
L.~H\"ormander, \textit{Quadratic hyperbolic operators}, Microlocal Analysis and Applications, LNM vol. 1495, L. Cattabriga, L. Rodino (Eds.) (1991), pp. 118--160.
%
\bibitem{Hormander2}
L.~H\"ormander, \textit{Symplectic  classification  of quadratic  forms,  and  general  Mehler
formulas}, Math. Z. \textbf{219} (3) (1995), 413--449.
%
\bibitem{Leray1}
J.~Leray, \textit{Lagrangian Analysis and Quantum Mechanics: A Mathematical Structure Related to Asymptotic Expansions and the Maslov Index}, The MIT Press, 1981.
%
\bibitem{Nicola1}
F.~Nicola and L.~Rodino, \textit{Global Pseudo-Differential Calculus on Euclidean Spaces}, Birkh\"auser, Basel, 2010.
%
\bibitem{Pazy1}
A.~Pazy, \textit{Semigroups of Linear Operators and Applications to Partial Differential Equations}, Applied Mathematical Sciences \textbf{44}, Springer-Verlag, New York Berlin Heidelberg Tokyo (1983).
%
\bibitem{PRW1}
K. Pravda-Starov, L.~Rodino and P.~Wahlberg, \textit{Propagation of Gabor singularities for Schr\"odinger equations with quadratic Hamiltonians}, Math. Nachr. \textbf{291} (1) (2018), 128-159.
%
\bibitem{Reed1}
M.~Reed and B.~Simon, \textit{Methods of Modern Mathematical Physics}, Vol. 1, Academic Press, 1980.
%
\bibitem{Rodino1}
L.~Rodino and P.~Wahlberg, \textit{The Gabor wave front set}, Monaths. Math. \textbf{173} (4) (2014), 625--655. 
%
\bibitem{SW2}
R.~Schulz and P.~Wahlberg, \textit{Microlocal properties of Shubin pseudodifferential and localization operators}, 
J. Pseudo-Differ. Oper. Appl. \textbf{7} (1) (2015), 91--111.
%
\bibitem{Shubin1}
M.~A.~Shubin, \textit{Pseudodifferential Operators and Spectral Theory}, Springer, 2001.
%
\bibitem{Tataru}
D.~Tataru, \textit{Phase space transforms and microlocal analysis}, Phase space analysis of partial differential equations, Vol. II, Pubbl. Cent. Ric. Mat. Ennio Giorgi, Scuola Norm. Sup., Pisa (2004), pp. 505--524.
%
\bibitem{Treves1}
F.~Treves, \textit{Topological Vector Spaces, Distributions and Kernels}, 
Academic Press, New York, San Francisco, London, 1967. 
%
\bibitem{Weinstein1}
A.~Weinstein, \textit{A symbol class for some Schr\"odinger equations on $\rr n$}, 
Amer. J. Math. \textbf{107} (1) (1985), 1--21. 
%

\end{thebibliography}

\end{document}